\documentclass[12pt,reqno]{amsart}
\usepackage{amssymb}

\usepackage{float}
\usepackage{graphicx}

\usepackage{enumerate}
\usepackage{float}
\usepackage{ifthen}

\usepackage{centernot}
\usepackage{mathtools}

\usepackage{mathtools}% Loads amsmath

\usepackage{graphicx}

\nonstopmode\numberwithin{equation}{section}

\usepackage[applemac]{inputenc}

\newtheorem{proposition}{Proposition}[section]
\newtheorem{definition}{Definition}[section]
\newtheorem{theorem}{Theorem}[section]

\newtheorem{lemma}{Lemma}[section]

\newcommand{\be}{\begin{equation}}
\newcommand{\ee}{\end{equation}}

\newcommand{\R}{\mathbb{ R}}

\allowdisplaybreaks
    
\begin{document}
\title{On some collinear configurations  in the planar three-body problem}

\author{
Alexei Tsygvintsev \\
}
{
\noindent 
UMPA, Ecole Normale Sup\'{e}rieure de Lyon\\
Email: alexei.tsygvintsev@ens-lyon.fr
}

\bigskip
\begin{abstract}
In this paper, we further investigate the planar Newtonian three-body problem with a focus on collinear configurations, where either the three bodies or their velocities are aligned. We provide an independent proof of Montgomery's result \cite{M1}, stating that apart from the Lagrange's solution, all negative energy solutions to the zero angular momentum case result in syzygies, i.e., collinear configurations of positions. The concept of generalised syzygies, inclusive of velocity alignments, was previously explored by the author for bounded solutions in \cite{AT1}. In this study, we broaden our scope to encompass negative energy cases and provide new bounds. Our methodology builds upon the elementary Sturm-Liouville theory and the Wintner-Conley “linear” form of the three-body problem, as previously explored in the works of Albouy and Chenciner \cite{V1,V2,V3}.
\end{abstract}

\keywords{   dynamical systems, celestial mechanics, three-body problem, syzygies}
\maketitle

\section{Introduction} 

We consider the motion of three points with masses $m_1$, $m_2$, and $m_3$, denoted $P_1$, $P_2$, and $P_3$, respectively, in the plane. These points have positions $(x_i, y_i)\in\mathbb{R}^2$, where $i=1,2,3$.

The Newtonian three-body problem \cite{W} involves finding the motion of three point masses under the influence of their mutual gravitational attraction.

The equations of  motion can be written in the complex compact form as follows  
\begin{equation} \label{PLC}
\ddot z_1=m_2 \frac{z_{21}}{| z_{21} |^3}- m_3 \frac{z_{13}}{| z_{13} |^3}, \quad \ddot z_2=m_3 \frac{z_{32}}{| z_{32} |^3}- m_1 \frac{z_{21}}{| z_{21} |^3}, \quad   \ddot z_3=m_1 \frac{z_{13}}{| z_{13} |^3}- m_2 \frac{z_{32}}{| z_{32} |^3}\,,
\end{equation} 
where $z_k= x_k+i y_k\in \mathbb{C}$, $k=1,2,3$ and $ z_{kl}=z_k-z_l$.
We  assume  that the total linear momentum  is  zero:
\begin{equation} \label{sums}
\sum\limits _km_k \dot z_k=\sum\limits _km_k z_k=0\,, 
\end{equation}
by placing  the centre of mass  at the origin of the  coordinate system.

The word {\it syzygy} ( from Late Latin s$\mathrm{   \bar{y}}   $zygia =“conjunction” )  historically been used by astronomers to describe the alignment of celestial bodies, and in this context, it refers to a configuration where all three points lie on a straight line.  In general, a solution of the $N$-body problem is said to have a syzygy at  $t=t_0$ if at that moment all bodies belong to a certain straight line.     

Let $ t \mapsto  z_i(t)$, $i=1,2,3$ be any solution of the  equations \eqref{PLC} defined  for  $t\in I =[0, a)$, $a>0$. In our work \cite{AT1} we proposed to study the natural generalisation of syzygies by adding the 
supplementary condition of collinearity of velocities.

\begin{definition} \label{D}
The three bodies $P_1,P_2,P_3$ form a generalised syzygy  at the moment $t_0\in I$  if  at least one of the complex triplets $ (z_1,z_2,z_3)(t_0)$ (positions)  or   $ (\dot z_1,\dot z_2,\dot z_3)(t_0)$ (velocities)  belongs to the same straight line passing through the origin (see Figure \ref{fpicm1}).
\end{definition}

\vspace*{-15mm}
 \begin{figure}[htbp] 
 \includegraphics[width=0.8\linewidth ]{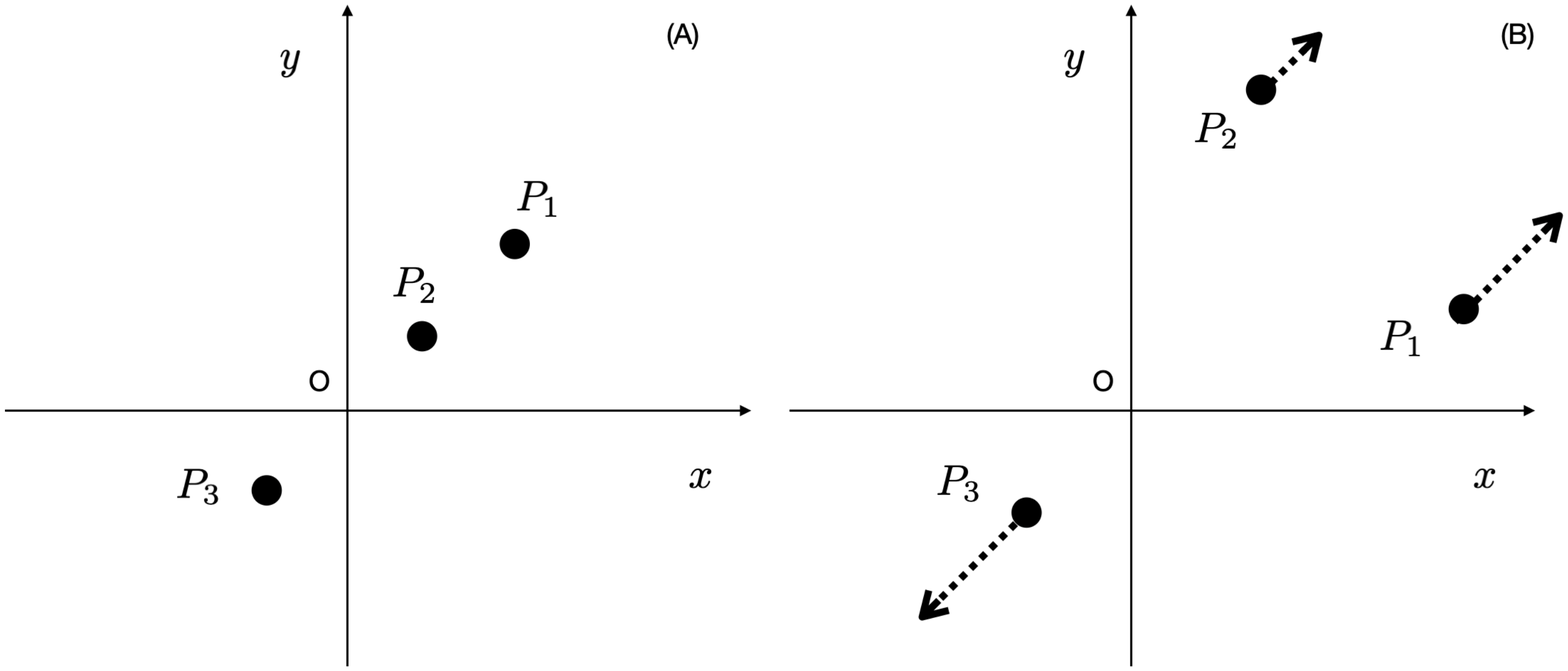}
 \vspace*{-15mm}
 \caption{ A  generalised syzygy: (A) -- the positions are collinear (eclipse), (B) -- the velocities (dashed arrows)  are collinear }
 \label{fpicm1}
 \end{figure}

In his work \cite{M1}, Montgomery demonstrated that aside from the Lagrange solution, every solution to the Newtonian three-body problem with zero angular momentum and negative energy inevitably encounters a syzygy, defined as a collinear configuration of positions.  Non-syzygy solutions with nonzero angular momentum are still  relatively unexplored.   
Diacu demonstrated in \cite{Diacu} that the set of initial conditions leading to syzygy solutions in the planar three-body problem is non-empty and open. 

 In particular, the orbit of one particle crosses the line of the other two and can not be
tangent to this line in the transition point. Generalised syzygies may occur more frequently, as they include collinear configurations of velocities as well.

The paper is organised as follows.

In Chapter 2, we provide a new independent  proof of Montgomery's result \cite{M1} on the existence of syzygies in the zero angular momentum case. 
The general approach involves formulating a second-order “inear” matrix equation $\ddot X = A X$ (see \cite{V1,V2,V3} for more detail), where the matrix $X$ characterises the configuration. The term “linear” is used since the matrix A also depends on the configuration. However, despite this dependence, we can draw certain conclusions based on the general properties of $A$.

In Chapter 3, we impose a certain algebraic restriction on the mutual distances between three bodies. In the case of periodic behaviour, this restriction guarantees the existence of a syzygy.

Finally, in Chapter 4, we present a simple algebraic condition that defines an open set of initial conditions leading to a generalised syzygy, which holds for arbitrary angular momentum.

\section {Zero Angular Momentum  Case }

Let  $\Gamma\,:\,  t\mapsto  (z_1(t),z_2(t),z_3(t))$ be a   zero angular momentum solution with negative energy  of the three-body problem \eqref{PLC}. 
 In this section we prove that if $\Gamma$  is collision free for a sufficiently long period of time, it will always encounter  a syzygy. 
A similar result was first established by Montgomery in \cite{M1} using the shape sphere approach. Our proof is essentially algebraic and is more elementary.

Before proceeding further, we need to establish some preliminary results.

After introducing  the new variables $w_i=m_i z_i$, $i=1,2,3$  the  relations  \eqref{sums}  yield
\begin{equation}  \label{bnb}
 \sum \limits _i   w_i= \sum\limits _i \dot w_i=0.
\end{equation}
Writing 
\begin{equation} \label{1234}
w_k=X_k+iY_k, \quad  X_k=m_k x_k, \quad   Y_k=m_k y_k,  \quad k=1,2,3\,, 
\end{equation}
and using  \eqref{bnb}, one derives from the equations \eqref{PLC} the following $2\times 2$  matrix  equation 
\begin{equation} \label{EQ}
\ddot X= AX, \quad X=
\left [
\begin{array}{ll}
X_1 & Y_1 \\
X_2 & Y_2
\end{array}
\right ], \quad  
A=\left [ 
\begin{array}{lllll}
-m_2 \rho_3 -m_{13} \rho_2 &  & m_1\rho_{32} \\
m_2 \rho_{31}&  & -m_1\rho_3 -m_{32} \rho_1
\end{array} 
\right]\,,
\end{equation} 
where 
\begin{equation}
\rho_1=1/ |z_{32}|^3,  \quad  \rho_2=1/ |z_{13}|^3,  \quad \rho_3=1/ |z_{21}|^3,  \quad m_{ij}=m_i+m_j,  \quad \rho_{ij}=\rho_i-\rho_j\,.
\end{equation}

The matrix $A$ is related to the  Wintner-Conley endomorphism encoding the forces.  Its higher dimension version was investigated in \cite{V1,V2,V3}  in  the study of $n$-body configurations and  the same equation \eqref{EQ} was  used (up to reduction, transposition and some scaling factor).

\begin{theorem}  \label{T2}
Let   $ t\mapsto  (z_1(t),z_2(t),z_3(t))$, $t \in [0,T_1]$  be a zero angular momentum collision-free solution  to the three-body problem \eqref{PLC}  with  negative energy $H=-\alpha$, $\alpha>0$ where 
\begin{equation} \label{eqq}
T_1(\alpha)= \frac{\sqrt{2} \pi \Sigma}{\alpha^{3/2}}\,,
\end{equation}
and
\begin{equation} \label{FO}
\Sigma=\frac{(m_3 m_2)^{3/2}} { m_{32}^{1/2}  } +    \frac{(m_1 m_3)^{3/2}} { m_{13}^{1/2}  }   +    \frac{(m_2 m_1)^{3/2}} { m_{21}^{1/2}  }  \,.
\end{equation}
Then there exists $t_0  \in[0, T_1]$ such that the three bodies form a syzygy at the moment $t=t_0$.
\end{theorem} 
\begin{proof}

Introducing $x =[X_1,X_2]^T$, $ y=[Y_1,Y_2]^T$ and using the equation  \eqref{EQ}, we obtain  $\ddot x= Ax$,  $\ddot y=Ay$. Therefore,  $\ddot \Delta_1$ can be written as follows
\begin{equation} \label{l}
\ddot \Delta_1= \det([Ax,y])+\det([x,Ay])+2\det([\dot x, \dot y])=\det([Ax,y])+\det([x,Ay])+2\Delta_2\,.
\end{equation}
A simple algebraic computation shows that $\det([Ax,y])+\det([x,Ay])=\mathrm{Tr}(A)\Delta_1$.

Thus,  as follows from \eqref{l},
\begin{equation} \label{eqdf}
\ddot \Delta_1= \mathrm{Tr}(A)\Delta_1+ 2\Delta_2\,,
\end{equation} 
where
\begin{equation}
 \mathrm{Tr}(A)=-(m_{32}\rho_1+m_{13}\rho_2+m_{21}\rho_3)\,.
\end{equation}

%%%%%%%%%%%%%%%%%%%%

\begin{lemma} \label{lemmm1}
Let $ t \mapsto  z_i(t)$, $i=1,2,3$  be  any solution of the three-body problem \eqref{PLC} with negative energy $H=-\alpha$, where $\alpha>0$. Then  $\mathrm{Tr}(A)\leq - \alpha^3/ \Sigma^2$ along this solution  with  $\Sigma$ defined  by  \eqref{FO}.
\end{lemma}
%%%%%%%%%
\begin{proof}
By introducing  $r_i= \rho_i^{1/3}$, $i=1,2,3$  the total energy of the three-body problem \eqref{PLC} can be written as
\begin{equation}
H= K-U(r), \quad  U(r)= m_3m_2 r_1+ m_1 m_3 r_2+ m_2m_1 r_3, \quad r=(r_1,r_2,r_3)\,,
\end{equation}
where $K\geq 0$ is the kinetic energy.

Since $K-U=-\alpha$ and $K\geq 0$,  one obtains $U\geq \alpha$.

Writing $F(r)=-\mathrm{Tr}(A)(r)= m_{32} r_1^3+m_{13} r_2^3+m_{21} r_3^3$, we shall  calculate, for any $s>0$,  the minimum of  the function $r \mapsto F(r)$ on the compact set
\begin{equation}
K_s=\{  r \in \mathbb R^3 \, | \,  U(r)=s,  r_i\geq 0, i=1,2,3\}\,.
\end{equation}
Because  $K_s$ is   convex (triangle)   and $F$ is a convex function on $K_s$ (since $d^2 F\geq 0$ on $K$),  it is sufficient to determine its  local minimum.  
One  computes: 
\begin{equation} 
\nabla U=( m_3 m_2, m_1 m_3, m_2 m_1), \quad  \nabla F=3 ( m_{32} r_1^2, m_{13} r_2^2, m_{21} r_3^2)\,.
\end{equation} 
The Lagrange  multiplier $\lambda$ found from the equations $ \nabla F= \lambda\,  \nabla U$, $U=s$ is given by 
\begin{equation}
\lambda=  \frac{3 s^2}{\Sigma^2}\,,
\end{equation} 
with $\Sigma$  defined   in \eqref{FO}.

The corresponding extremum point  is 
\begin{equation} 
r^*=(r_1^*, r_2^*, r_3^*)= \frac{s}{\Sigma} \left( \sqrt{ \frac{ m_3m_2 }{ m_{32}}},  \sqrt{\frac{m_1m_3}{ m_{13}}}, \sqrt{\frac{m_2 m_1}{ m_{21}}} \right)\in K_s\,,
\end{equation} 
which is a local minimum of $F$ because $d^2F(r^*)$ is  positive definite  and $d^2U=0$.
Thus, by substitution:
\begin{equation}
\min _{r \in K_s}  \,  F(r) =  F(r^*)=\frac{s^3}{\Sigma^2}\,.
\end{equation} 
Hence, considering  $s\geq \alpha$,   one shows that $U(r)\geq \alpha$ implies $ F (r)=- \mathrm{Tr}(A)(r) \geq \displaystyle \frac{\alpha^3}{ \Sigma^2}$. The proof of the Lemma \ref{lemmm1} is finished. 
\end{proof} 
%%%%%%%%%%%

%%% start of input 

The function $t\mapsto C(t)=\dot X X^{-1} $ satisfies, using \eqref{EQ},   the following matrix Riccati equation 
\begin{equation}\label{RIC}
\dot C+ C^2=A\,.
\end{equation} 
Combined with the Cayley-Hamilton identity
\begin{equation} 
 C^2-\mathrm{Tr}(C) C+\det(C)I_2=0\,,
 \end{equation}
 the equation \eqref{RIC} yields 
\begin{equation} \label{aq}
\dot C+ \mathrm{Tr}(C)C=A+\det(C)I_2 \,.
\end{equation}
Applying Liouville's formula to the equation $\dot X = CX$, we obtain:
\begin{equation} \label{Trace1}
 \dot \Delta_1=  \mathrm{Tr}(C)\Delta_1\,.
\end{equation}

Multiplying both sides of equation \eqref{aq} by $\Delta_1$, we get
\begin{equation} \label{xxt}
\frac{d}{dt} \left(  \Delta_1 C  \right)= \Delta_1 A+ \Delta_1 \det(C) I_2= \Delta_1 A+ \Delta_2 I_2\,,
\end{equation} 
where we have  used  $\det(C)=\Delta_2 / \Delta_1$. 

Introducing the adjugate matrix  $\tilde X$ of $X$ and using $ C=\dot X X^{-1}$,  we compute

\begin{equation} \label{Trace2}
\Delta_1 C =\dot X \tilde X=
\left [
\begin{array}{lll}
\dot X_1 & \dot Y_1 \\
\dot X_2 & \dot Y_2
\end{array}
\right ] \, \left [
\begin{array}{rrr}
Y_2 & -Y_1  \\
-X_2 & X_1
\end{array}
\right ]\,.
\end{equation}

The matrix  $A(\rho_1,\rho_2,\rho_3)$, defined  in \eqref{EQ} can be written as a linear combination 
\begin{equation} \label{fff1}
A=\rho_1A_1+\rho_2 A_2+\rho_3 A_3 \,,
\end{equation}
where $A_1$, $A_2$, and $A_3$ are $2\times 2$ constant matrices that depend on the masses and are defined by
\begin{equation} \label{M3}
A_1=\left [
\begin{array}{rrr}
0 & 0 \\
-m_2 & -m_{32}
\end{array}
\right ] , \, A_2=\left [
\begin{array}{rrr}
-m_{13} & -m_1 \\
0 & 0
\end{array}
\right ], \,A_3=\left [
\begin{array}{rr}
-m_2 & m_1 \\
m_2& -m_1
\end{array}
\right ]\,.
\end{equation}

Let $E = M_3(\mathbb R)$ be the Euclidean space of real $2\times 2$ matrices equipped with the inner product $<\,, \,>$ defined by $<A,B> = \mathrm{Tr}(A^TB)$ for $A,B\in E$.

We check that  $A_1,A_2,A_3 \in E$ are linearly independent  and span a $3$-dimensional vector subspace of   $E$. Furthermore, we have the identity
\begin{equation} \label{fff2}
-\frac{A_1+A_2+A_3}{M}=  I_2, \quad M=\sum m_i\,.
\end{equation}

Combining   equations \eqref{fff1},  \eqref{fff2}  and  \eqref{xxt}, we obtain
\begin{equation} \label{meqs}
\frac{d}{dt} (\dot X \tilde X)=\left (\Delta_1 \rho_1- \frac{\Delta_2}{M} \right)A_1+\left (\Delta_1 \rho_2- \frac{\Delta_2}{M}\right)A_2+\left (\Delta_1 \rho_3- \frac{\Delta_2}{M}\right)A_3\,.
\end{equation}

Let us check, with help of  the equation \eqref{meqs},  constancy of the angular momentum.  We introduce the matrix 
\begin{equation}
L=\left [
\begin{array}{lrr}
- m_3^{-1} & m_1^{-1}+m_3^{-1} \\ \\
-m_2^{-1} -m_3^{-1} & m_3^{-1}
\end{array}
\right ] \,.
\end{equation} 

It can be  checked that  $ <L,A_i>=0$, $i=1,2,3$. 

Taking the inner product of both sides of the equation \eqref{meqs} with $L$, we find that
\begin{equation}
\frac{d}{dt} \left < (\dot X \tilde X), L      \right>=0\,,
\end{equation} 
i.e. the function  $I=\left < \dot X \tilde X, L  \right> =const$ is time independent. 

One can easily verify   that $I$, written as a function  of variables $X_i$,  $Y_j$, $ 1\leq i,j \leq 2$ (by excluding $(X_3,Y_3)$ with help of \eqref{bnb})   is the angular  momentum of the three-body problem \eqref{PLC} :
\begin{equation}
I= \sum_{i=1}^3 m_i \, R_i \times \dot R_i= \sum_{i=1}^3 \frac{S_i}{m_i}\,,
\end{equation} 
where $R_i=(x_i,y_i)$, $S_i=X_i\dot Y_i- \dot X_i Y_i$,  $ 1\leq i \leq 3$.

%%%%%% end of input 
The equation \eqref{xxt} can be written, by replacing $\Delta_2$ using equation \eqref{eqdf}, in the following form
\begin{equation}
\frac{d} { dt} \left ( \dot X\tilde X- \frac{\dot \Delta_1}{2} I_2 \right ) = \Delta_1 \left (A- \frac{\mathrm{Tr}(A)}{2} I_2\right)\,. 
\end{equation} 

Applying \eqref{fff1}, we can transform it to
\begin{equation} \label{MMM3}
\frac{d} { dt} \left ( \dot X\tilde X- \frac{\dot \Delta_1}{2} I_2 \right )= \Delta_1 \sum_{i=1}^3 \rho_i \tilde A_i\,,
\end{equation}
where $\tilde A_i= A_i- \displaystyle \frac{\mathrm{Tr}(A_i)}{2} I_2$,  $ \mathrm{Tr}(\tilde A_i)=0$,   $1\leq i \leq 3$.

It can be shown, using  \eqref{M3},   \eqref{fff2},  that  $\tilde A_1+\tilde A_2+\tilde A_3=0$.  Substituting  $\tilde A_3=-\tilde A_1-\tilde A_2$ in \eqref{MMM3} we obtain finally 
\begin{equation} \label{finalreduced}
\frac{d} { dt} \left ( \dot X\tilde X- \frac{\dot \Delta_1}{2} I_2 \right )=\Delta_1 (\rho_1-\rho_3) \tilde A_1+\Delta_1 (\rho_2-\rho_3) \tilde A_2\,,
\end{equation}
where
\begin{equation} 
\tilde A_1=\left [
\begin{array}{rrr}
m_{32}/2 & 0 \\
-m_2 & -m_{32}/2
\end{array}
\right ] , \, \tilde A_2=\left [
\begin{array}{rrr}
-m_{13}/2 & -m_1 \\
0 & m_{13}/2
\end{array}
\right ], \quad m_{ij}=m_i+m_j\,.
\end{equation}

Moreover, as follows from  \eqref{Trace1}, \eqref{Trace2}, one always has
\begin{equation} \label{trdet}
\mathrm{Tr}(\dot X  \tilde X) =\dot \Delta_1, \quad  \mathrm{det}(\dot X \tilde  X)= \Delta_1 \Delta_2\,.
\end{equation}
Let $I=\left < \dot X \tilde X, L  \right> =k$  where $k\in \mathbb R$ is the constant value of the angular momentum. 
Computing $\left  < \dot X \tilde X, L  \right> $  using  \eqref{Trace2}, it is easy to derive the following  equality:  
\begin{equation} \label{form}
\dot X\tilde X- \frac{\dot \Delta_1}{2} I_2= \frac{b}{m_2} \tilde A_1- \frac{a}{m_1} \tilde A_2- k\frac{m_3 }{2} J\,,
\end{equation}
where
\begin{equation}
a=
\left |
\begin{array}{lll}
X_1 & Y_1\\
\dot X_1 & \dot Y_1
\end{array}
\right |, \quad b= \left |
\begin{array}{lll}
X_2 & Y_2\\
\dot X_2 & \dot Y_2
\end{array}
\right |, \quad J=\left [
\begin{array}{rrr}
1 & 0\\
0 & -1
\end{array}
\right ]\,.
\end{equation}

For  $k=0$, the formula \eqref{form} is given by 
\begin{equation} \label{longer}
\dot X\tilde X- \frac{\dot \Delta_1}{2} I_2= \frac{b}{m_2} \tilde A_1- \frac{a}{m_1} \tilde A_2=
\left [
\begin{array}{rrr}
\beta a+ \gamma b &  a\\
-b  &- \beta a -\gamma b 
\end{array}
\right ]=R
\,,
\end{equation}
where
\begin{equation} \label{alphab}
 \beta=\frac{1}{2}  \left (  \frac{m_3}{m_1}+1  \right) , \quad \gamma=\frac{1}{2} \left (  \frac{m_3}{m_2}+1  \right)  \,,
\end{equation} 
and
\begin{equation} \label{qform}
-\mathrm{det}(R)=\beta^2 a^2 +(2 \beta \gamma -1)ab+\gamma^2 b^2\,.
\end{equation}
As follows from \eqref{alphab}, one always has  $\beta, \gamma >1/2$. Using these inequalities, it is easy to show  that the quadratic form  in  \eqref{qform}, as a function of  $a$ and $b$,  is positive  and therefore 
\begin{equation}
\mathrm{det}(R)(a,b)\leq 0, \quad \forall \, a,b \in \mathbb{R} \,.
\end{equation}
Since $\mathrm{Tr}(R)=0$, the eigenvalues of $R$ are real.  Hence, according to \eqref{longer}, the eigenvalues $\lambda_1$, $\lambda_2$  of $\dot X \tilde X$ are also real.
Thus, using  \eqref{trdet},  we have the following inequality holding for arbitrary $a$, $b\in \R$
\begin{equation} \label{deltain}
\mathrm{Tr}(\dot X  \tilde X)^2-4 \,  \mathrm{det}(\dot X \tilde  X)     =\dot \Delta_1^2 -4 \Delta_1 \Delta_2\geq 0 \,.
\end{equation}
Suppose that the solution $ t\mapsto  (z_1(t),z_2(t),z_3(t))$  is syzygy free in the interval $[0,T_1]$. In particular, we can suppose 
\begin{equation}\label{hypot}
\Delta_1(t)>0, \quad   \forall t\in [0,T_1]\,,
\end{equation}
so  that the function $\delta :\, t  \mapsto  \sqrt{   \Delta_1(t)  }$  is differentiable in this interval. 

Differentiating $\delta$ twice  using  \eqref{eqdf}, we obtain
\begin{equation} \label{rt1}
\ddot \delta = \eta \delta, \quad \eta = \left(\frac {\mathrm{Tr}(A) }{2}- \frac{\dot \Delta_1^2-4\Delta_1 \Delta_2}{4\Delta_1^2}    \right), \quad \quad t\in [0,T_1] \,.
\end{equation}
According to \eqref{deltain} and Lemma  \ref{lemmm1}:
\begin{equation}
\eta(t)\leq \frac{\mathrm{Tr}(A) }{2}\leq -\zeta^2,  \quad \zeta^2= \frac{\alpha^3}{2\Sigma^2}, \quad t\in [0,T_1]\,.
\end{equation}

As follows from the zero comparison theorem of the Sturm-Liouville theory  \cite{C}, the solution  $\delta$  of \eqref{rt1}  always has a zero  between any two consecutive zeros of any  solution $y$ of the  equation $\ddot y =-\zeta^2 y$ whose general solution is  $y(t)=A \cos( \zeta t+\phi_0)$,   $A,\phi_0\in \mathbb{R}$.    For a  nonzero $A$,   every two consecutive zeros  of  $y$  are separated by an interval of the length $\pi/\zeta$ . 
Hence,  $\delta(t_0)=\Delta_1(t_0)=0$ for some  $0<t_0< T_1=\pi/\zeta$ with $T_1$  given by the formula \eqref{eqq}. 
This contradicts our hypothesis  \eqref{hypot}. The proof of Theorem  \ref{T2} is finished.

\end{proof}

\section{ Existence of syzygies for periodic solutions }

In this section we will show that any periodic solution of the three-body problem will have a syzygy, i.e. exhibits a collinear configuration of the bodies  as soon as some geometrical constraints are imposed on the shape of the triangle formed by the bodies. Not every periodic solution has a syzygy; for example, the Lagrange periodic solution has the three bodies situated at the vertices of an equilateral triangle and never aligned.
 \begin{definition} \label{DEF22}
 A periodic solution of the three-body  problem \eqref{PLC}  is  called $\theta$-rigid  if there exists a non-zero vector  $\theta=(\theta_1,\theta_2,\theta_3)\in \mathbb R^3$   such that
 \begin{equation} \label{theta}
 \theta_1( \rho_3(t)- \rho_2(t) )  +\theta_2 (\rho_1(t)-\rho_3(t)) +\theta_3 (\rho_2(t)- \rho_1(t)) \geq 0, \quad  \forall \, t\in \mathbb R\,,
  \end{equation} 
 where the sum is strictly positive for at least one  $t_0\in \mathbb R$.
   \end{definition} 

For instance, let us consider a periodic solution   for which 
\begin{equation} \label{ineq23}
|z_{32}(t)|>|z_{13}(t)|  \Leftrightarrow \rho_1(t) < \rho_2(t), \quad \forall t \in \mathbb R\,.
\end{equation}
Thus, it is $\theta$-rigid with respect to the vector $(\theta_1,\theta_2,\theta_3)=(0,0,1)$.   

An example of this is Euler's collinear periodic solution, in which all bodies are perpetually collinear  and each body describes an elliptical orbit, implying that a syzygy occurs at every moment.
 
 We note that the  Lagrange equilateral solution is not $\theta$-rigid for any choice of $\theta$. Indeed, in this case $\rho_1=\rho_2=\rho_3$ and, therefore, the sum in \eqref{theta} is always zero.

One might ask whether it is possible to have a periodic solution without syzygies, in which one side of the triangle is always smaller than the other and thus satisfies one of the conditions of the form \eqref{ineq23}.

Our next theorem provides a negative response to this question.

%%%%%%
\begin{theorem} \label{Theorem2}
Every $\theta$-rigid periodic solution to the three-body problem \eqref{PLC}  admits  a syzygy.
\end{theorem}
%%%%%%%%
\begin{proof}
Let  $t\mapsto z_i(t)$, $i=1,2,3$  be a periodic solution of the three-body problem \eqref{PLC}  with period $\tau>0$.  

Let us suppose that it is $\theta$-rigid  and has no syzygies i.e. $\Delta_1(t)=\det(X(t)) \neq 0$, $\forall t \in [0,\tau]$.  Without loss of generality, we can assume that 
\begin{equation} \label{sigma}
\Delta_1(t) >0, \quad  \forall t \in [0,\tau]\,.
\end{equation}
Let $S_i=X_i\dot Y_i- \dot X_i Y_i$, $i=1,2,3$  be  three oriented areas of parallelograms formed by the vectors $(X_i,Y_i)$ and $(\dot X_i,\dot Y_i)$. Then, using \eqref{EQ} and  \eqref{bnb}, one derives the following equations
\begin{equation} \label{systemof3}
\dot S_1=m_1 \Delta_1 ( \rho_3-\rho_2),  \quad  \dot S_2=m_2 \Delta_1 ( \rho_1-\rho_3), \quad \dot S_3=m_3 \Delta_1 ( \rho_2-\rho_1)\,.
\end{equation}
Thus, 
\begin{equation} \label{INT}
\frac{d} {dt} \left ( \sum _{i=1}^3 \frac{\theta_i  S_i}{m_i}   \right)= \Delta_1 S , \quad S=   \theta_1( \rho_3- \rho_2 )  +\theta_2 (\rho_1-\rho_3) +\theta_3 (\rho_2- \rho_1) , \quad  t \in [0, \tau]\,.
\end{equation}
Integrating  \eqref{INT} and using the periodicity of $S_i$, $i=1,2,3$ we find
\begin{equation}
\int _0 ^{\tau} \Delta_1 S \, dt=0\,,
\end{equation}
which obviously contradicts  \eqref{sigma} and that $S(t_0)>0$ for some $t_0\in [0,\tau]$.  The proof is finished.
\end{proof}
%%%%%%% end of proof here

As pointed out by Richard  Montgomery to the author, in the case of the $\theta=(0,0,1)$ rigid solution, the corresponding orbit, when mapped to the shape sphere \cite{M2}, belongs to the half-sphere bounded by the isosceles circle $ |z_{32}| = |z_{13}|$. According to Theorem \ref{Theorem2}, it should thus always intersect the collinear plane of the shape space.

\section{Sufficient condition for the  existence of  generalised syzygies}

In this section, we revisit a simple geometric condition -- first introduced in our previous study \cite{AT1} -- which, based on initial positions and velocities, guarantees the occurrence of generalised syzygies (see Definition \ref{D}).

 \begin{definition} \label{D1}
The configuration of  the bodies $P_1$, $P_2$ and  $P_3$ at the moment $t_0\in I$  is called  {\it antisymmetric},  if the oriented areas of parallelograms  spanned by the vectors 
$(z_j(t_0),z_k(t_0))$ and $(\dot z_j(t_0), \dot z_k(t_0))$ are  nonzero and have  opposite signs for some $j\neq k$ (see Figure \ref{fpic0}).   This condition is equivalent algebraically to
\begin{equation} \label{cond}
 \left  |
\begin{array}{lll}
x_j &  y_j\\
x_k & y_k
\end{array}
\right|(t_0)  \cdot  \left  |
\begin{array}{lll}
\dot x_j &  \dot y_j\\
\dot x_k & \dot y_k
\end{array}
\right|(t_0)= (x_j y_k -y_j x_k)(t_0) \cdot (\dot x_j \dot y_k -\dot y_j \dot x_k)(t_0)<0.
\end{equation}
\end{definition} 

\begin{figure}[htbp] 
 \includegraphics[width=0.4\linewidth]{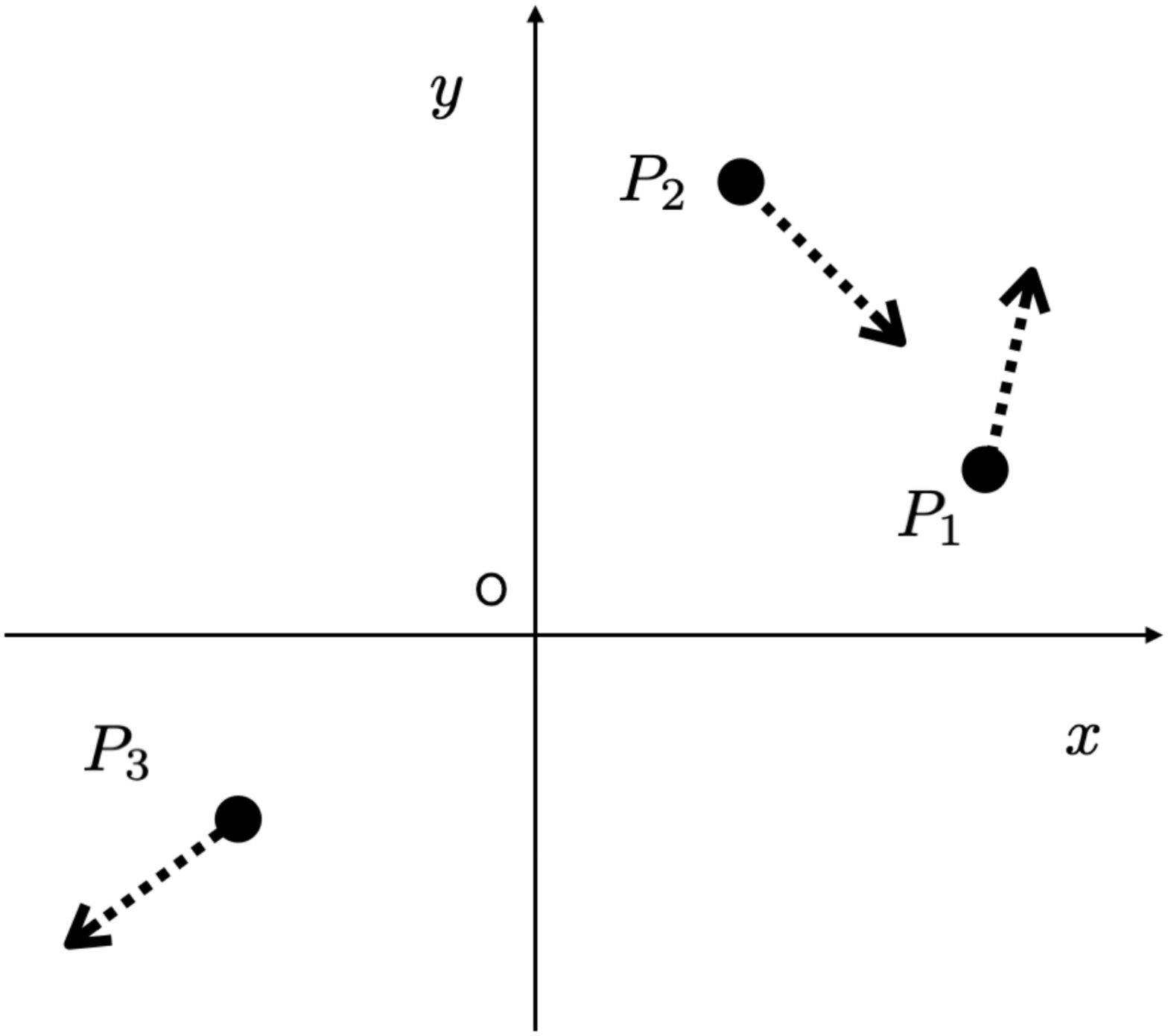}
\vspace*{-3mm}
 \caption{ Antisymmetric configuration of three bodies}
 \label{fpic0}
 \end{figure}
%%%%%%
\begin{proposition}\label{R}
The  property of being antisymmetric is independent of  the choice of the pair of the bodies: once verified for a particular pair of  bodies $P_j$, $P_k$, $j\neq k$,
the condition \eqref{cond} will be also satisfied for all other possible  choices of pairs $P_n$, $P_m$, $n\neq m$.
\end{proposition} 

\begin{proof}
We observe that the condition \eqref{cond} is invariant  under the  permutation of $j$ and $k$.  
Without loss of generality, we can assume that  \eqref{cond} holds for  $(j,k)=(1,2)$ and show  that it is true also for the case where $(j,k)=(2,3)$.  The proof for other cases is similar.

The product of two determinants in \eqref{cond}  for $j=2$, $k=3$ can be written as  
 \begin{equation} \label{formulas}
 \delta=
  \underbrace{
 \left  |
\begin{array}{lll}
x_2 &  y_2\\
x_3 & y_3
\end{array}
\right |
\cdot  \left  |
\begin{array}{lll}
\dot x_2 &  \dot y_2\\
\dot x_3 & \dot y_3
\end{array}
\right|} _a=\displaystyle  \frac{1}{ m_2^2 m_3^2}
\left|
\begin{array}{lll}
m_2 x_2 & m_2 y_2\\
m_3 x_3 &m_3  y_3
\end{array}
\right| \cdot \left  |
\begin{array}{lll}
m_2 \dot x_2 &  m_2 \dot y_2\\
m_3 \dot x_3 &m_3 \dot y_3
\end{array}
\right| \,.
\end{equation}

As can be seen from the equalities \eqref{sums}:
\begin{equation}\label{rt}
\begin{array}{llll}
[m_3 x_3, m_3 y_3]=- [m_1 x_1, m_1 y_1]-[m_2 x_2, m_2 y_2], \\ 

 [m_3 \dot x_3, m_3 \dot y_3]=- [m_1 \dot x_1, m_1 \dot y_1]-[m_2 \dot x_2, m_2 \dot y_2]\,.
\end{array}
\end{equation} 

Therefore, by using \eqref{formulas}, \eqref{rt}, and the elementary properties of determinants, we have:
\begin{equation} \label{r}
\delta = \displaystyle  \frac{1}{ m_2^2 m_3^2}
\left|
\begin{array}{lll}
m_2 x_2 & m_2 y_2\\
-m_1 x_1 &-m_1  y_1
\end{array}
\right| \cdot  \left  |
\begin{array}{lll}
m_2 \dot x_2 &  m_2 \dot y_2\\
-m_1 \dot x_1 &-m_1 \dot y_1
\end{array}
\right|= \displaystyle  \frac{m_1^2}{m_3^2}  \underbrace{\left  |
\begin{array}{lll}
x_1 &  y_1\\
x_2 & y_2
\end{array}
\right |
\cdot  \left  |
\begin{array}{lll}
\dot x_1 &  \dot y_1\\
\dot x_2 & \dot y_2
\end{array}
\right| }_b\end{equation}
Obviously, the sign of $a$ in  \eqref{formulas}  and $b$ in   \eqref{r} are the same. That finishes the proof.
\end{proof}

\begin{proposition}
Both  definitions  of a  generalised  syzygy   and an antisymmetric configuration,   are  invariant when  replacing $(x_i,y_i)$ by $(X_i,Y_i)$ and $(\dot x_i, \dot y_i)$ by $(\dot X_i, \dot Y_i)$ for $i=1,2,3$. 
\end{proposition} 
\begin{proof}
It follows  immediately from  \eqref{1234}  and the  identity below, verified  for any $j\neq k$: 

\begin{equation} 
\left |
\begin{array}{llll}
x_j  & y_j \\
x_k &  y_k
\end{array}
\right| \cdot 
\left |
\begin{array}{llll}
\dot x_j  & \dot y_j \\
\dot x_k &  \dot y_k
\end{array}
\right|=\frac{1}{ m_j^2 m_k^2}
\left |
\begin{array}{llll}
X_j  & Y_j \\
X_k &  Y_k
\end{array}
\right| \cdot 
\left |
\begin{array}{llll}
\dot X_j  & \dot Y_j \\
\dot X_k &  \dot Y_k
\end{array}
\right| \,.
\end{equation} 
\end{proof}

The author has shown in \cite{AT1} that the existence of a generalised syzygy can be guaranteed by the antisymmetry condition \eqref{cond}, provided that the mutual distances of the three bodies remain bounded.
The following theorem generalises this result to the case of negative energy, which includes unbounded (escape) solutions in particular. Moreover, the proof presented here is simpler and more straightforward.

%%%%%%%%%
\begin{theorem} \label{T} 
Let   $ t\mapsto  (z_1(t),z_2(t),z_3(t))$  be a solution of the three-body problem \eqref{PLC}  with negative energy $H=-\alpha$, $\alpha>0$. We assume that the initial  configuration of the bodies $P_1$, $P_2$,  $P_3$ at $t=0$ is antisymmetric 
and the solution is collision free for $t\in [0,T]$ where 
\begin{equation}   \label{FO1}
T(\alpha)=\frac{   \pi   \Sigma } {    \alpha^{3/2}  } \,,
\end{equation}
and $\Sigma$ is defined by \eqref{FO}. 
Then there exists $t\in [0, T]$ such that the three bodies attain a generalised syzygy at time  $t=t_0$.
\end{theorem} 
%%%%%%%%%%
\begin{proof}

Writing $d=\Delta_2 / \Delta_1$,  the equation  \eqref{eqdf}  can be transformed into  the second-order  Hill's linear differential equation:
\begin{equation} \label{yu}
\ddot \Delta_1= (\mathrm{Tr}(A)+ 2d ) \Delta_1\,.
\end{equation}

%%%%%%%%%%
 We suppose now  that the solution of \eqref{PLC} $ t \mapsto  z_i(t)$, $i=1,2,3$ is defined in the interval  $[0,T]$, with $T$ given in \eqref{FO1}.  
Assuming it starts at $t=0$  from an antisymmetric configuration and that no generalised syzygy happens for any $0<t \leq T$ we obtain 
\begin{equation} \label{ddd}
d(t)=\frac{\Delta_2}{\Delta_1}(t)<0, \quad \Delta_1(t)  \neq 0, \quad  \forall t\in [0, T]\,.
\end{equation}

Therefore,  according to  \eqref{yu}, \eqref{ddd} and the Lemma  \ref{lemmm1}:
\begin{equation}  \label{equation} 
\ddot \Delta_1= \phi  \Delta_1, \quad \phi=  \mathrm{Tr}(A)+2d   \leq -\theta^2, \quad \theta^2=\alpha^3/ \Sigma^2,   \quad t\in [0, T]\,. 
\end{equation}

Comparing  \eqref{yu}  with $\ddot y=-\theta^2 y$, and using the same Sturm-Liouville argument as in the proof of Theorem \ref{T2}, we conclude that $\Delta_1$ admits at least one zero in the interval $[ 0, \pi / \theta]=[0,T]$. This contradicts our hypothesis  \eqref{ddd}. The proof of Theorem  \ref{T} is finished.
\end{proof}

%%%%%%%%%%%%%

\section{Conclusion}

Montgomery \cite{M1} showed the existence of syzygies in the three-body problem for the case of zero angular momentum and negative energy, except for the Lagrange homothetic solutions.
Our Theorem \ref{T} is free from the restriction on the angular momentum and   contains both an easy-to-check sufficient condition and an upper bound on the time instant when the generalised syzygy occurs.
Theorem \ref{Theorem2} deals with the collinear configurations (syzygies) in the periodic case.  Under the assumption  that the triangle formed by the bodies obeys a geometric restriction called $\theta$-rigidity, we show that the bodies become aligned at some instant, resulting in a syzygy in the corresponding solution.  We also conjecture that a similar result holds for bounded non-periodic solutions.
A more in-depth analysis of equations \eqref{meqs} and \eqref{RIC} may reveal additional interesting properties regarding collinear configurations in the three-body problem.\\

\noindent {\bf Akcnowledgments}\\
\noindent I am deeply grateful to Alain Chenciner  and Richard Montgomery  for their invaluable suggestions and insightful remarks.

\end{document}